\documentclass[12pt,reqno]{article}

\usepackage[usenames]{color}
\usepackage[dvipsnames]{xcolor}
\usepackage{amssymb}
\usepackage{amsmath}
\usepackage{amsthm}
\usepackage{amsfonts}
\usepackage{amscd}
\usepackage{graphicx}
\usepackage{algorithm}
\usepackage{algorithmic}

\usepackage[colorlinks=true,
linkcolor=webgreen,
filecolor=webbrown,
citecolor=webgreen]{hyperref}

\definecolor{webgreen}{rgb}{0,.5,0}
\definecolor{webbrown}{rgb}{.6,0,0}

\usepackage{color}
\usepackage{fullpage}
\usepackage{float}

\usepackage{graphics}
\usepackage{latexsym}
\usepackage{epsf}

\usepackage{enumerate,tikz,listings,array,longtable}
\usetikzlibrary{shapes,arrows,chains}
\usepackage{pgf}

\DeclareMathOperator{\Flatten}{Flatten}
\DeclareMathOperator{\LrMax}{LrMax}

\DeclareMathOperator{\lrmax}{lrmax}
\DeclareMathOperator{\LwMp}{WLrMax}

\DeclareMathOperator{\lwmp}{wlrmax}
\DeclareMathOperator{\Rlmin}{RlMin}
\DeclareMathOperator{\RGF}{RGF}
\DeclareMathOperator{\blocks}{blocks}
\DeclareMathOperator{\run}{run}
\DeclareMathOperator{\rlmin}{rlmin}

\newcommand{\seqnum}[1]{\href{https://oeis.org/#1}{\rm \underline{#1}}}

\begin{document}

\theoremstyle{plain}
\newtheorem{theorem}{Theorem}
\newtheorem{corollary}[theorem]{Corollary}
\newtheorem{lemma}[theorem]{Lemma}
\newtheorem{proposition}[theorem]{Proposition}

\theoremstyle{definition}
\newtheorem{definition}[theorem]{Definition}
\newtheorem{example}[theorem]{Example}
\newtheorem{conjecture}[theorem]{Conjecture}

\theoremstyle{remark}
\newtheorem{remark}[theorem]{Remark}

	\begin{center}
		{\large \bf Merging-Free Partitions and Run-Sorted Permutations}\\
		\bigskip
		Fufa Beyene\footnote{Corresponding author.}\\
		Department of Mathematics, Addis Ababa University\\
		P.O. Box 1176\\ Addis Ababa, Ethiopia\\
			\href{mailto: fufa.beyene@aau.edu.et}{\tt fufa.beyene@aau.edu.et}\\
			\ \\
		Roberto Mantaci	\\
		IRIF, Universit\'e de Paris\\8, Place Aur\'elie Nemours, 75013\\ Paris, France\\
			\href{mailto: mantaci@irif.fr}{\tt mantaci@irif.fr}
	\end{center}
	
	\setcounter{page}{1} \thispagestyle{empty}
	
	\begin{abstract}
	In this paper, we study merging-free partitions with their canonical forms and run-sorted permutations. We give a combinatorial proof of the conjecture made by Nabawanda et al. We describe the distribution of the statistics of runs and right-to-left minima over the set of run-sorted permutations and we give the exponential generating function for their joint distribution. We show the number of right-to-left minima is given by the shifted distribution of the Stirling number of the second kind. We also prove that the non-crossing merging-free partitions are enumerated by powers of $2$. We use one of the constructive  proofs  given in the paper to implement an algorithm for the exhaustive generation  of run-sorted permutations by number of runs.	
\end{abstract}

	\section{Introduction}
	Given a non-empty finite subset $A$ of positive integers, a \emph{set partition} $P$ of $A$ is a collection of disjoint non-empty subsets $A_i$ called \emph{blocks} of $A$ such that $\cup_{i=1}^kA_i=A$ \cite{Bo,Ma}.
	We shall use the notation $[n]:=\{1, 2, \ldots, n\}$, where $n$ is a fixed positive integer. It is well known that set partitions over $[n]$ and set partitions over $[n]$ having $k$ blocks are counted by the Bell numbers, $b_n$ and Stirling numbers of the second kind, $S(n,k)$ respectively (\cite{Bo,Ro,St1}). 
	Mansour \cite{Ma} defined the \emph{block representation} of a set partition where the elements in a block are arranged increasingly and the blocks are arranged in increasing order of their first elements. Mansour also gave a way to encode a set partition (in its block representation) by its \emph{canonical form}, that is, every integer is encoded by the  number of the block it belongs to. We note that canonical forms of set partitions coincide with the so-called \emph{restricted growth functions} ($\RGF$).	
	
	Callan \cite{Ca} introduced the ``flattening'' operation ($\Flatten$) on set partitions, which acts in such a way that a permutation $\sigma$ is obtained from a set partition $P$ by removing the separators enclosing the different blocks of $P$ in its block representation.
	For example, the set partition $P=126/3/48/57$ is in the block representation, and so we remove the separators $``/"$ and obtain the permutation $\sigma=12634857$. As a result of Callan's work, such objects are getting the attention of different researchers and several new findings are emerging, for example see  \cite{Al-Na,Ma-Sh-Wa,Na-Ra}. 
	
	In the literature permutations obtained this way are sometimes called ``flattened partitions''. We found this term somewhat confusing because these objects are permutations and not partitions, consequently, since the runs of the resulting permutations are sorted by the increasing values of their respective minima, we chose to adopt the term {\em  run-sorted permutations} already used by Alexandersson and Nabawanda \cite{Al-Na}. Run-sorted permutations are counted by the shifted Bell numbers (see \cite{Na-Ra}).

	The same permutation can be obtained by flattening several set partitions. For instance, the permutation $\sigma=12634857$ can also be  obtained by flattening the set partition $P'=126/348/57$. Among all the set partitions having the same $\Flatten$, we will distinguish the only one whose number of blocks is the same as the number of runs of the permutation obtained by flattening it (this is the set partition $P'$ for the permutation $\sigma$). For obvious reasons we  named these objects {\em merging-free partitions}. The $\Flatten$ operation clearly becomes injective and hence a bijection if restricted to the set of merging-free partitions.
	
	In this article we study some properties of run-sorted permutations as well as of merging-free partitions and their canonical forms, we compute the distribution of some statistics (runs, right-to-left minima, $\ldots$) over these sets, we relate these classes to the classes of separated partitions and of non-crossing partitions and we provide an exhaustive generation algorithm for the class of run-sorted permutations partitioned by number of runs. In particular, in Section \ref{secflatrun}, we give the characterization of the canonical forms of merging-free partitions, and show that they can be bijectively related to $\RGF$s of one size smaller.
	
	In Section \ref{sec3}, we give a combinatorial bijective proof of a recurrence relation in Theorem \ref{rrNaRa} satisfied by run-sorted permutations over $[n]$ having $k$ runs, a recurrence relation that was conjectured by Nabawanda et al.\ \cite{Na-Ra}. We also give the interpretation of the proof of the same result by working on the canonical forms of merging-free partitions. 
	
	In Section \ref{secrlmin}, we prove that the distribution of right-to-left minima over run-sorted permutations is the same as the distribution of the number of blocks over set partitions of one size smaller (and also given by the shifted Stirling number of the second kind). We refine the recurrence relation satisfied by the number of run-sorted permutations over $[n]$ having $k$ runs by counting these permutations by number of runs and by number of right-to-left minima simultaneously and we obtain an exponential generating function for the associated three-variables formal series. Munagi \cite{Mu} proved that the set partitions over $[n]$ having $k$ blocks such that no two consecutive integers in the same block are also counted by the shifted Stirling numbers of the second kind. So, in this section we also show that these partitions bijectively correspond to run-sorted permutations over $[n]$ having $k$ right-to-left minima. 
	
	Non-crossing partitions are Catalan enumerated objects introduced in the founding work of Becker \cite{Be} and later deeply studied by different eminent scholars like Kreweras and Simion (\cite{Kr, Si}). We characterize the class of non-crossing merging-free partitions in Section \ref{secNoncross} and we enumerate them according to their number of blocks and show that the total number of such partitions is counted by the power of $2$. 
	
	Finally, Section \ref{secalgo} presents an exhaustive generation algorithm for run-sorted permutations partitioned by the number of runs, based on the recurrence relation proved in Theorem \ref{rrNaRa} and using the classical dynamic programming techniques.
	
	\subsection{Definitions, Notation, and Preliminaries} 
	\begin{definition}
		A \textit{set partition} $P$ of $[n]$ is defined as a collection $B_1, \ldots, B_k$ of nonempty disjoint subsets such that $\cup_{i=1}^kB_i=[n]$. The subsets $B_i$ will be referred to as ``blocks''. 
	\end{definition}
	\begin{definition}
		A set partition $P=B_1/\cdots/B_k$ is said to be the \emph{block representation} of $P$ if the blocks $B_1, \ldots, B_k$ are sorted in such way that $\min (B_1) < \min (B_2) <\cdots < \min (B_k)$ and the elements of every block are arranged in  increasing order.
		\label{def1}
	\end{definition}
	We will always write set partitions in their block representation. Let $\mathcal{SP}(n)$ denote the set of all set partitions over $[n]$ and $b_n=|\mathcal{SP}(n)|$ the $n$-th Bell number.
	\begin{definition} 
		The \emph{canonical form} of a set partition of $[n]$ is a $n$-tuple indicating the block in which each integer occurs, that is, $f= f_1f_2\cdots f_n$ such that $j\in B_{f_j}$ for all $j$ with $1\leq j\leq n$. 
	\end{definition}
	\begin{example}
		If $P=138/2/47/56\in\mathcal{SP}(n)$, then its canonical form is $f=12134431$.
	\end{example}
	\begin{definition}
		A \emph{restricted growth function} $(\RGF)$ over $[n]$ is a function $f :[n]\mapsto[n]$, where $f=f_1\cdots f_n$ such that $f_1=1$ and $f_i \leq 1+\max\{f_1,\ldots, f_{i-1}\}$ for $2\leq i\leq n$, or equivalently, such that the set $\{f_1, f_2, \ldots, f_i\}$ is an integer interval for all $i\in[n]$. 
	\end{definition}
	The canonical forms of set partitions are exactly the restricted growth functions ($\RGF$). We let $\RGF(n)$ denote the set of all restricted growth functions over $[n]$. We will note $f\in \RGF(n)$ as a word $f_1f_2\cdots f_n$ over the alphabet $[n]$, where $f_i=f(i)$.
	We define the statistic of the set of left-to-right maxima of $f$ by $$\LrMax(f)=\{i~ : f_i>f_j, 1\leq i \leq n, j<i\},$$
	and a statistic of the set of weak left-to-right maxima of $f$ by $$\LwMp(f)=\{i :~f_i\geq f_j, 1\leq i \leq n, j<i\},$$ We also use the notation $\lrmax(f):=|\LrMax(f)|$ and $\lwmp(f):=|\LwMp(f)|$.
	\begin{example} 
		If $f=121132342\in\RGF(9)$, then observe that $\LrMax(f)=\{1, 2, 5, 8\}$ and $\LwMp(f)=\{1, 2, 5, 7, 8\}$.
	\end{example}
	\begin{definition}
		Let $P=B_1/\cdots/B_k$  be a set partition. We say that $P$ is merging-free if $\max(B_i)>\min(B_{i+1}), 1\leq i\leq k-1$. 
	\end{definition}
	A permutation $\pi$ over $[n]$	will be represented in the one-line notation, $\pi=\pi_1\pi_2\cdots\pi_n$. In particular, every permutation can be considered as a word of length $n$, with letters in $[n]$. We define the set of right-to-left minima of $\pi$ by $$\Rlmin(\pi)=\{\pi_i~: ~\pi_i<\pi_j, j>i\},$$ and we use the notation $\rlmin(\pi):=|\Rlmin(\pi)|$.
	\begin{definition}
		A maximal increasing subsequence of consecutive letters in the word of a permutation $\pi$ is called a run.
	\end{definition}
	\begin{definition}
		\emph{Flattening} a set partition is an operation by which we obtain a permutation from the set partition $P=B_1/\cdots/B_k$ over $[n]$ by concatenating its blocks. We denote the resulting permutation by $\Flatten(P)=\pi$. 
	\end{definition}
	If a permutation is obtained by flattening a set partition, then its runs are ordered in such a way that the minimum of the runs are increasing, therefore, we will call all permutations in $\Flatten(\mathcal{SP}_n)$ \emph{run-sorted permutations}. We let $\mathcal{RSP}(n)$ denote the set of all run-sorted permutations over $[n]$ and $r_n$ is its cardinality. 
\begin{remark}
	Merging-free partitions over $[n]$ and run-sorted permutations over $[n]$ are in bijection, because the restriction of $\Flatten$ to the merging-free partitions is a bijection.
\end{remark}
 Nabawanda et al.\ \cite{Na-Ra} proved the following result.
 \begin{proposition} 
		\label{prop0}
		The number of set partitions over $[n]$ and the number of run-sorted permutations over $[n+1]$ (and therefore, the number of merging-free partitions over $[n+1]$) are equal. That is, $r_{n+1}=b_n$ for all $n\geq1$.\end{proposition}
	\begin{proof} We give a sketch of the proof. Let $P=B_1/\cdots/B_k\in\mathcal{SP}(n)$.
	The corresponding run-sorted permutation in $\mathcal{RSP}(n+1)$ is constructed as follows: move each minimum element of the block at the end of its block, remove the slashes, increase every integer by 1, and finally attach the integer 1 at the front.
	Conversely, we construct the set partition over $[n]$ corresponding to a run-sorted permutation $\pi\in \mathcal{RSP}(n+1)$ as follows. Put a slash after each right-to-left minimum of $\pi$, then delete the integer $1$, and decrease every integer by 1, finally arrange the elements of each block in increasing order.\end{proof}
	\begin{example}
		If $P=14/258/37/6\in\mathcal{SP}(8)$, then by the above operation we obtain the run-sorted permutation $\pi=152693847\in \mathcal{RSP}(9)$.
		Conversely, for $\pi=152693847\in \mathcal{RSP}(9)$ we have $\Rlmin(\pi)=\{1, 2, 3, 4, 7\}$. So by putting a slash after each right-to-left minimum we obtain $1/52/693/84/7\longrightarrow14/258/37/6=P$.
	\end{example}
	\section{Canonical forms of merging-free partitions}
	\label{secflatrun}
	In this section we characterize the $\RGF$s  corresponding to merging-free partitions and we present some results related to these canonical forms.
\begin{remark}
	Let $f=f_1\cdots f_n$ be the canonical form of a set partition $P=B_1/\cdots/B_k$ over $[n]$ having $k$ blocks. We have $i\in \LrMax(f)$ if and only if $i=\min(B_{f_i})$.
\end{remark}
	\begin{proposition}
		\label{CharSubFlat}
		There is a bijection between the set of merging-free partitions over $[n]$ and the set $T_n$ of $\RGF$s $f=f_1f_2\cdots f_n$ over $[n]$ satisfying the condition that every left-to-right maximum letter $s>1$ of $f$ has at least one occurrence of $s-1$ on its right.
	\end{proposition}
	\begin{proof}
		If $P=B_1/B_2/\cdots /B_k$ is a merging-free partition with $k$ blocks, then $\min(B_{s-1})<\min(B_{s})$ and $\max(B_{s-1})>\min(B_s), ~s=2, \ldots, k$. Note that every leftmost occurrence of a letter in $f$ is a left-to-right maximum letter. The positions of the leftmost and rightmost occurrences of the letter $s$ in $f$ correspond to the minimum and the maximum elements of the block $B_s$, respectively. Thus, if $1<s\leq k$, then $f_{\min(B_s)}=s$ and ~$f_{\max(B_{s-1})}=s-1$.
	\end{proof}
	\begin{definition}
	Let $f=f_1f_2\cdots f_n\in \RGF(n)$. If the occurrence of the letter $f_i$ in $f$ has no repetition, then we say that $f_i$ is unique in $f$, that is, $f_i$ is unique if and only if $i$ forms a singleton block in the partition. A weak left-to-right maximum $i$ in $f$ for which there exists $i_0<i$ such that $f_i=f_{i_0}$ is called a non-strict left-to-right maximum.
	\end{definition} 
	We shall give a combinatorial proof of Proposition \ref{prop0} in terms of canonical forms. 
	For each $i\in[n]$ we define $u_i$ as the number of unique left-to-right maximum letters of $f$ in the positions $1, \ldots, i-1$ that are smaller than $f_i$. We let $u=(u_1, \ldots, u_n)$. Let $\delta=(\delta_1, \ldots, \delta_n)$, where
	\begin{displaymath}\delta_i=
	\begin{cases}
		1, &\text{ if  $f_i$ is a non-unique left-to-right maximum letter of  $f$;}\\
		0, &\text{ otherwise}.
	\end{cases}\end{displaymath}
	
	Define a mapping $\alpha: \RGF(n)\mapsto T_{n+1}$, where $T_{n+1}$ is the set of canonical forms of merging-free partitions over $[n+1]$, by $\alpha(f)=1\cdot f'$, that is, a concatenation of $1$ and $f'$, where $f'=f'_1\cdots f'_n$ is obtained from $f$ as follows:
	\begin{equation*}
		f'=f-u+\delta.\end{equation*}
	\begin{example}
		If $f=1213124$, then $f_1=1$ and $f_2=2$ are non-unique left-to-right maximum letters, while $f_4=3$ and $f_7=4$ are the unique ones. So, $u=(0, 0, 0, 0, 0, 0, 1)$ and $\delta=(1, 1, 0, 0, 0, 0, 0)$. Thus, $ f'=f-u+\delta=2313123$ and $\alpha(f)=1\cdot f'=12313123\in T_8$.
	\end{example}	
\begin{lemma}\label{lemmaltrmax}
	If $f\in \RGF(n)$ and $f'$ is obtained from $f$ as in the above construction, then $\LrMax(f)\subseteq\LwMp(f')$.
\end{lemma}
\begin{proof}
	Let $\LrMax(f)=\{i_1, \ldots i_k\}$.
	We proceed by induction on $k$. For $k=1$, note that $i_1=1$ and $f_1'=f_1-u_1+\delta_1=f_1+\delta_1\in\{1,2\}$. Thus, the assertion is true for the basis step. Suppose that the assertion is true for $k-1$, and we show that $i_k$ is a weak left-to-right maximum in $f'$. By definition we have
	$$u_{i_k}=\begin{cases} u_{i_{k-1}}+1, &\text{ if } f_{i_{k-1}} \text{ is unique;}\\u_{i_{k-1}}, &\text{ if } f_{i_{k-1}} \text{ is non-unique.}
	\end{cases} $$ 
	If $f_{i_k}$ is unique, then $\delta_{i_k}=0$ and in either cases we have \begin{align*}f_{i_k}'=f_{i_k}-u_{i_k}=f_{i_{k-1}}'.
\end{align*} 
If $f_{i_k}$ is non-unique, then $\delta_{i_k}=1$ and in either cases we have \begin{align*}f_{i_k}'=f_{i_k}-u_{i_k}+1=f_{i_{k-1}}'+1.
\end{align*} 
Therefore, in all of the above cases we have $f_{i_k}'\geq f_{i_{k-1}}'$. For the intermediate values, we already know them to be non-weak left-to-right maximum letters, hence they are not greater than $f_{i_{k-1}}'$. Thus, by using the induction hypothesis we see that $i_k$ is a weak left-to-right maximum in $f'$. \end{proof}
\begin{lemma}
	For all $f\in \RGF(n)$ we have that $\alpha(f)\in T_{n+1}$.
\end{lemma}
\begin{proof}
	By the above lemma it is easy to see that $1\cdot f'\in\RGF(n+1)$. Let $f'=f_1'f_2'\cdots f_n'$ and let $f_i'>1$ be a left-to-right maximum letter in $1\cdot f'$, then $f_i$ is a non-unique left-to-right maximum letter in $f$. This implies that there is some $j>i$ such that $f_i=f_j$, and hence $u_i=u_j, \delta_j=0$. So, $\alpha(f(i))=f'_i=f_i-u_i+1=f_j'+1$. Therefore, every left-to-right maximum letter $s>1$ of $1\cdot f'$ has some occurrence $s-1$ from its right and hence $1\cdot f'\in T_{n+1}$. \end{proof}
	
	We now define a map $\beta
	: T_{n+1}\mapsto \RGF(n)$ which associates each $1\cdot g=1\cdot g_1\cdots g_n\in T_{n+1}$ with a function $\beta(1\cdot g)=g'=g'_1g'_2\cdots g'_n$, where $g'$ is obtained from $g$ as follows. For each $i\in [n]$, we let $v_i$ denote the number of non-strict left-to-right maximum letters in $g$ that are less than or equal to $g_i$ in the positions $1, \ldots, i-1$. Let $v=(v_1, \ldots, v_n)$. Further, let $\delta'=(\delta'_1, \ldots, \delta'_n)$, where
	\begin{displaymath} \delta'_i=
	\begin{cases}
		1, &\text{ if $g_i$ is left-to-right maximum of $g$};\\
		0, &\text{ otherwise}.
	\end{cases}\end{displaymath}
	Then, $g'$ is obtained from $g$ as follows:
	$$ g'=g+v-\delta'.$$
	For instance, if $1\cdot g=122134321\in T_9$, then $g=22134321, v=(0, 0, 0, 1, 1, 1, 1, 0)$, and $\delta'=(1, 0, 0, 1, 1, 0, 0, 0)$. Thus, $g'=12134431\in\RGF(8)$. Note that $\beta=\alpha^{-1}$ and	as a result, we have the following proposition.
	\begin{proposition}
		The mapping $\alpha$ from the set $\RGF(n)$ to the set $T_{n+1}$ is a bijection. \qed
	\end{proposition}

\begin{corollary}\label{corltmax}
	If $f\in \RGF(n)$ and $\alpha(f)=1\cdot f'$, then $\LrMax(f)=\LwMp(f')$. \qed\end{corollary}
	We now evaluate the number $l_n$ of $f\in\RGF(n)$ having the sequence $u=0$, that is, if $f=f_1f_2\cdots f_n$, then for each $i\in[n]$ there is no unique left-to-right maximum letter smaller than $f_i$ on its left. The set partitions corresponding to such functions are exactly those satisfying the condition that their blocks have size at least two except for the last block, which may be singleton. The sequence of the numbers $l_n, n\geq0$ is the same as the OEIS \seqnum{A346771}.
	\begin{theorem}
		For all $n\geq1$ we have 
		\begin{equation}\label{thmzerou}
		l_n=\sum_{k=1}^{n-1} {n-1 \choose k} l_{n-k-1}, l_0=l_1=1.
		\end{equation}
	\end{theorem}
\begin{proof}
	Let $f\in\RGF(n)$ satisfy the above condition. Since $f_1=1$ is the smallest integer, every such function has at least two $1$s. Suppose that $f$ has $k+1$ occurrences of $1$s. If we delete all the $1$s and decrease each of the remaining integers by $1$, then we obtain a $\RGF$ over $[n-k-1]$, with the same condition as $f$. So, there are $l_{n-k-1}$ such functions. We now  choose $k$ positions from $\{2, 3, \ldots, n\}$ where $1$ can be inserted, and this is possible in $\displaystyle {n-1 \choose k}$ ways. Therefore, by applying the product rule and then taking the sum over all possible $k$ we have the right hand side of (\ref{thmzerou}).
\end{proof}
	\section{Run distribution in run-sorted permutations}
	\label{sec3}
	The following table presents the first few values $r_{n, k}$ of the number of run-sorted permutations over $[n]$ having $k$ runs (see \seqnum{A124324}).
	\begin{table}[h!]
		\centering
		\begin{tabular}{c|ccccccc} 
			$n\backslash k$ & 1 & 2 & 3 & 4 & 5 & 6 & 7\\
			\hline 1 & 0 & & & & & &\\
			2 & 1 & 0 & & & & &\\
			3 & 1 & 1 & 0 & & & &\\
			4 & 1 & 4 & 0 & 0 & & &\\
			5 & 1 & 11 & 3 & 0 & 0 & &\\
			6 & 1& 26 & 25 & 0 & 0 & 0 &\\
			7 & 1 & 57 & 130 & 15 & 0& 0 & 0\\
			\hline 
		\end{tabular}
		\caption{The values of $r_{n, k}$ for $1\leq k,n\leq 7$}
	\end{table}
	
	\begin{remark}
		The number $k$ of runs of a run-sorted permutation over $[n]$ satisfies the condition $$1\leq k\leq\lceil\frac{n}{2}\rceil,~ n\geq1,$$
		because each run except the last has length at least $2$.
	\end{remark}
	
	The following result was conjectured by Nabawanda et al.\ \cite{Na-Ra}, who also gave a justification of the first term of the right-hand side of (\ref{eq1}). We were the first to provide a combinatorial bijective proof justifying the second term and thus prove the conjecture. We show a complete combinatorial proof of the conjecture here using our bijection.
	
	\begin{theorem}
		\label{rrNaRa}
		The number $r_{n, k}$ of run-sorted permutations of $[n]$ having $k$ runs satisfies the recurrence relation
		\begin{equation}\label{eq1}
			r_{n, k}=kr_{n-1, k}+(n-2)r_{n-2, k-1},~~n\geq2, ~k\geq1,\end{equation}
		where $r_{0, 0}=1, ~r_{1, 0}=0, ~r_{1, 1}=1$.
	\end{theorem}
	In order to prove this result, we partition the set $\mathcal{RSP}(n,k)$ of run-sorted permutations over $[n]$ having $k$ runs into two subsets:
	$\mathcal{RSP}^{ (1)}(n,k)$ and $\mathcal{RSP}^{ (2)}(n,k)$, where $\mathcal{RSP}^{ (1)}(n,k)$ is the set of elements of $\mathcal{RSP}(n,k)$ in which the removal of the integer $n$ does not decrease the number of runs and	$\mathcal{RSP}^{ (2)}(n,k)$ is the set of elements of $\mathcal{RSP}(n,k)$ in which the removal of the integer $n$ decreases the number of runs, this happens when the integer $n$ occurs between two integers $x$ and $y$ with $x<y$. For example, $12435\in \mathcal{RSP}^{(1)}(5, 2)$ and $15234\in \mathcal{RSP}^{(2)}(5, 2)$.	We will denote the cardinalities of these subsets by $r^{ (1)}_{n,k}$  and $r^{(2)}_{n,k}$, respectively. 
	
	Let $\phi: [k]\times\mathcal{RSP}(n-1,k)\mapsto\mathcal{RSP}^{ (1)}(n,k)$	associating each element $(i,\sigma)\in[k]\times\mathcal{RSP}(n-1,k)$ with the permutation $\sigma'=\phi(i,\sigma)$ obtained from $\pi$ by inserting $n$ at the end of the $i$-th run of the permutation $\sigma$. It is easy to see that $\phi$ is a bijection (see \cite{Na-Ra}, p.\ 6).
	
	We now define the mapping $\psi:[n-2]\times \mathcal{RSP}(n-2, k-1)\mapsto \mathcal{RSP}^{(2)}(n, k)$, associating each element $(i, \pi)\in [n-2]\times \mathcal{RSP}(n-2, k-1)$ with the permutation $\pi'=\psi(i, \pi)$ obtained from $\pi$ by increasing all integers greater than $i$ by 1 and inserting the subword $n~ i{+}1$ immediately after the rightmost of the integers of the set $\{1, 2, \ldots, i\}$. 
	\begin{example} Let $i=3$ and $\pi=13524\in \mathcal{RSP}(5, 2)$. We construct $\psi(i, \pi)$ as follows: increase each integer greater than 3 in $\pi$ by $1$ to get $13625$, then insert the subword $7~(3{+}1)=7~4$ into the position after the rightmost of the integers $1, 2, 3$, thus the subword must be inserted between $2$ and $5$, hence, $\psi(3, 13524)=\pi'=1362745\in \mathcal{RSP}^{(2)}(7, 3)$.
	\end{example}
	\begin{lemma}
		\label{lemmacase2}
		For all $(i, \pi)\in [n-2]\times \mathcal{RSP}(n-2, k-1)$, we have $\psi(i,\pi)\in \mathcal{RSP}^{(2)}(n,k)$.
	\end{lemma}
	\begin{proof}
		Since $\pi\in \mathcal{RSP}(n-2,k-1)$ and the procedure inserts the subword $n~i{+}1$ immediately after the rightmost integer of the set $\{1, \ldots, i\}$, all integers to the right of $i+1$ are greater than $i+1$ and $i+1$ is the first element of a new run. Thus the resulting permutation is run-sorted with the number of runs increased by $1$. Furthermore, in the resulting permutation the integer $n$ is immediately preceded by some integer in the set $\{1, \ldots, i\}$ and immediately followed by $i+1$, hence its removal decreases the number of runs, so $\pi'\in \mathcal{RSP}^{(2)}(n,k)$.
	\end{proof}
	\begin{proposition}
		\label{bijectionpsi}
		The map $\psi$ defined above is a bijection.
	\end{proposition}
	\begin{proof}
		We prove that $\psi$ is both injective and surjective.	First let us assume that $(i_1, \pi_1)\neq (i_2, \pi_2)$ for $i_1, i_2\in[n-2]$ and $\pi_1, \pi_2\in \mathcal{RSP}(n-2, k-1)$. Let $\psi(i_1, \pi_1)=\pi_1'$ and $\psi(i_2, \pi_2)=\pi_2'$. Then $\pi_1'$ and $\pi_2'$ are run-sorted permutations in $\mathcal{RSP}^{(2)}(n,k)$ by the previous lemma. We consider two cases. If $i_1\neq i_2$, then in one of the two resulting permutations $n$ is followed by $i_1+1$ while in the other $n$ is followed by $i_2+1$. If $i_1=i_2$ and $\pi_1\neq \pi_2$, then the two run-sorted permutations $\pi_1$ and $\pi_2$ have at least two entries in which they differ. Thus inserting $n~i_1{+}1=n~i_2{+}1$ after the rightmost element of the set $\{1, 2, \ldots, i_1=i_2\}$, produces two different permutations $\pi_1'$ and $\pi_2'$.	Thus, in both cases, $\pi_1'=\psi(i_1, \pi_1)\neq\psi(i_2, \pi_2)=\pi_2'$, hence $\psi$ is injective. Next, consider any $\pi'\in \mathcal{RSP}^{(2)}(n, k)$, then $n$ does not appear in the last position. Let $j>1$ be the integer following $n$ in $\pi'$. We exhibit a pair $(i, \pi)\in [n-2]\times \mathcal{RSP}(n-2, k-1)$ such that $\psi(i, \pi)=\pi'$. Define $\pi$ to be the run-sorted permutation obtained from $\pi'$ by deleting the subword $n~j$ and by decreasing by 1 every integer greater than or equal to $j+1$ in the resulting word. Note that if $n$ follows the integer $i$ in $\pi'$, then $i<j$ and hence deleting the subword $n~ j$ from $\pi'$ reduces the number of runs by $1$ and the size of the partition by 2, whence $\pi\in \mathcal{RSP}(n-2, k-1)$ and $\psi(j+1, \pi)=\pi'$. Therefore, $\psi$ is a bijection. \end{proof}
	
	We are now ready to present the proof of Theorem \ref{rrNaRa}.
	\begin{proof}
		The left-hand side counts the number of run-sorted permutations in $\mathcal{RSP}(n, k)$.	The first term of the right-hand side counts the number of elements in $\mathcal{RSP}^{(1)}(n, k)$. Since $\phi$ is a bijection, we have $r_{n, k}^{(1)}=kr_{n-1, k}$. We show that the second term of the right-hand side counts the number of elements in $\mathcal{RSP}^{(2)}(n, k)$. By Proposition \ref{bijectionpsi} the sets $[n-2]\times \mathcal{RSP}(n-2, k-1)$ and $\mathcal{RSP}^{(2)}(n, k)$ have the same cardinality, that is, $(n-2)r_{n-2, k-1}=r_{n, k}^{(2)}$.
		Thus by combining the two parts we obtain $r_{n, k}=r_{n, k}^{(1)}+r_{n, k}^{(2)}$, and hence the recurrence relation in (\ref{eq1}).
	\end{proof}
	
	We also provide a bijective proof of Theorem \ref{rrNaRa} in terms of canonical forms. Let $T_{n,k}=\{f\in T_n : \lrmax(f)=k\}$, so $|T_{n,k}|=r_{n,k}$. Recall that $T_{n,k}$ is the set of the canonical forms of merging-free partitions over $[n]$ having $k$ blocks. 
	\begin{proof} 
		Firstly, if $f=f_1f_2\cdots f_{n-1}\in T_{n-1,k}$, then by concatenating any integer $i\in [k]$ at the end of $f$ we obtain a $f'\in T_{n,k}$. This is because, $f'$ satisfies the condition of Proposition \ref{CharSubFlat} if and only if $f$ does. This construction obviously produces $kr_{n-1,k}$ functions of $T_{n,k}$ having the property that by erasing the last value $f_n$ we obtain a function in $T_{n-1,k}$. 
		
		Secondly, if $f=f_1f_2\cdots f_{n-2}\in T_{n-2, k-1}$, let $i\in[n-2]$, and let $m=\max_{1\leq j\leq i}\{f_j \}$, then we construct $f'=f'_1f'_2\cdots f'_n\in T_{n,k}$  associated with $(i, f)$ as follows: increase by $1$ all $f_j$s such that $f_j\geq m, j>i$, insert $m+1$ at the position $i+1$, and append $m$ at the end. The functions obtained with the second construction are all different from those obtained using the former one. Indeed, by erasing the last integer from $f'$ we do not obtain a function in $T_{n-1,k}$. The reason is that the value $m+1$ in the position $i+1$ is a left-to-right maximum letter because of the choice of $m$. Now, by construction, the only occurrence of $m$ in $f'$ is at position $n$, by erasing this value the left-to-right maximum letter $m+1$ in the position $i+1$ is left without an occurrence of $m$ on its right. Therefore, $f_1'f_2'\cdots f_{n-1}'$ does not satisfy the property characterizing canonical forms of merging-free partitions. So, this contributes $(n-2)r_{n-2, k-1}$ to the number $r_{n,k}$ as there are $n-2$ possibilities for $i$ and the number of image values of $f'$ increases by $1$.
	\end{proof}
	\begin{example}
	Take $f=12132\in T_{5,3}$, and let $i=3$. We construct $(i,f)\mapsto f'$ as follows: we have $m=\max_{1\leq j\leq 3}\{f_j \}=\max_{1\leq j\leq 3}\{1, 2, 1 \}=2$, and $f'_1=f_1=1, f'_2=f_2=2, f'_3=f_3=1, f'_4=m+1=3, f'_5=f_4+1=4, f'_6=f_5+1=3, f'_7=m=2$. Thus, $f'=1213432\in T_{7,4}$.
\end{example}
\section{Right-to-left minima in  run-sorted permutations}
\subsection{The distribution of right-to-left minima over the set of run-sorted permutations}
	\label{secrlmin}	
	The following proposition gives us the relation between the statistics of right-to-left minima of run-sorted permutations and the weak left-to-right maxima of the canonical forms of the corresponding merging-free partitions.
	\begin{proposition}
		\label{rlmin-lrwmax}
		The set of right-to-left minima of a run-sorted permutation over $[n]$ and the set of weak left-to-right maxima of the canonical form of the corresponding merging-free partition are the same.
	\end{proposition}
	\begin{proof}
		Let $\Flatten(P)=\pi=\pi(1)\cdots\pi(n)\in \mathcal{RSP}(n)$, where $P$ is a merging-free partition over $[n]$. Let $f=f_1\cdots f_n$ be the canonical form of $P$ and let $\{i_1, \ldots, i_r\}$ be the set of the positions of the right-to-left minima of $\pi$, then by definition of right-to-left minima $\Rlmin(\pi)=\{1=\pi(i_1)<\pi(i_2)<\cdots<\pi(i_r)=\pi(n)\}$. Furthermore, if $1\leq j_1<j_2\leq r$ and we let $B_{q_1}$ be the block of $P$ containing $\pi(i_{j_1})$ and $B_{q_2}$ the block of $P$ containing $\pi(i_{j_2})$, then $q_1\leq q_2$ and by the definition of canonical form we have $f_{\pi(i_1)}\leq f_{\pi(i_2)}\leq\cdots\leq f_{\pi(i_r)}$. Assume that $\pi(j)\notin \Rlmin(\pi)$, then there exists some integer $s$ such that $\pi(j)>\pi(s), s>j$. Hence $f_{\pi(s)}>f_{\pi(j)}$ and $\pi(j)\notin \LwMp(f)$. Thus, $\{\pi(i_1), \ldots, \pi(i_k)\}\subseteq \LwMp(f)$.
		
		Conversely, if $\LwMp(f)=\{i_1, \ldots, i_s\}$, then for each $i_j$ we have $f_{i_j}\geq f_t, t<i_j$, that is, all integers $t<i_j$ belong either to $f_{i_j}$-th block or to a preceding block of $P$, therefore, in $\pi$ there is no integer smaller than $i_j$ on the right of $i_j$. Hence $i_j\in \Rlmin(\pi)$. Therefore, $\Rlmin(\pi)=\LwMp(f)$.
	\end{proof}
	\begin{example}
		If $P=149/238/57/6$, then its canonical form is $f=122134321$, and $\Flatten(P)=\pi=149238576$. Thus, we have $\Rlmin(\pi)=\{1, 2, 3, 5, 6\}=\LwMp(f)$.
	\end{example}
	Let $h_{n,r}$ denote the number of run-sorted permutations over $[n]$ having $r$ right-to-left minima.
	\begin{proposition}
		For all positive integers $n$ and $r$ with $2\leq r\leq n$ we have
		\begin{equation}\label{Recflatrlm} h_{n,r}=h_{n-1,r-1}+(r-1)h_{n-1, r}, ~ h_{1,1}=1.\end{equation}
	\end{proposition}
	\begin{proof}
		A run-sorted permutation $\pi'$ over $[n]$ can be obtained from a  run-sorted permutation $\pi$ over $[n-1]$ either by appending $n$ at its end, or by inserting $n$ before any of its right-to-left minima that is different from $1$. In the former case, the number of right-to-left minima increases by $1$, hence this contributes $h_{n-1,r-1}$ to the number $h_{n,r}$. In the later case, if $\Rlmin(\pi)=\{\pi(i_1), \pi(i_2), \ldots, \pi(i_r)\}$, then $1=\pi(i_1)<\pi(i_2)<\cdots<\pi(i_r)$. So, inserting $n$ before any $\pi(i_j)$ for $j\neq1$ makes $\pi(i_j)$ to be the minimum element of its run in $\pi'$. Thus the permutation $\pi'$ is run-sorted with the same number of right-to-left minima as $\pi$, and this contributes $(r-1)h_{n-1, r}$ as there are $r-1$ right-to-left minima different from $1$.
	\end{proof}
	We also give the interpretation of the bijective proof of the recursion formula in (\ref{Recflatrlm}) for the corresponding set of canonical forms of merging-free partitions using Proposition \ref{rlmin-lrwmax}. We interpret $h_{n,r}$ as the number of canonical forms in $T_n$ having $r$ weak left-to-right maxima, that is, $h_{n,r}=|\{f \in T_n : \lwmp(f) = r\}|$. All the elements of the set $\{f \in T_n :\lwmp(f) = r\}$ are obtained in a unique way
	\begin{enumerate}
		\item either from a $f=f_1f_2\cdots f_{n-1}\in T_{n-1,r-1}$ by concatenating $\max_{1\leq j\leq n-1}\{f_j\}$ at its end;
		\item or from a $f=f_1f_2\cdots f_{n-1}\in T_{n-1,r}$ with weak left-to-right maxima $\{i_1, i_2, \ldots, i_r\}$ as follows.
		For each $j=2, \ldots, r$ :
		\begin{itemize}
			\item[-] if $f_{i_j}$ is a non-strict left-to-right maximum letter of $f$, then increase by $1$ every integer $f_s$ such that $f_s\geq f_{i_j}$ and $s\geq i_j$, and
			\item[-] concatenate $f_{i_{j-1}}$ at the end of the resulting function.
		\end{itemize}
		\end{enumerate}
	Thus, the recurrence relation in (\ref{Recflatrlm}) follows.
	
	Recall that the recurrence relation satisfied by the Stirling numbers of the second kind is $S(n,r)=S(n-1,r-1)+rS(n-1,r)$. It is easy to see that from Corollary \ref{corltmax} of Section \ref{secflatrun} and Proposition \ref{rlmin-lrwmax}, the number of blocks in a set partition over $[n-1]$ is one less than the number of right-to-left minima of the corresponding run-sorted permutation over $[n]$ under the bijection in Proposition \ref{prop0}. So, the values of $h_{n,r}$ given in (\ref{Recflatrlm}) are the shifted values of the Stirling numbers of the second kind, that is, $h_{n,r}=S(n-1,r-1)$, for all $n\geq r\geq 1$. 

\subsection{The joint distribution of \texorpdfstring{$\run$}{r} and \texorpdfstring{$\rlmin$}{rl} over the set of run-sorted permutations}
	The statistics $\run$ and $\rlmin$ of a run-sorted permutation are obviously related. In particular, each minimum element of a run is always a right-to-left minimum, so $\run(\pi)\leq\rlmin(\pi), \forall\pi\in\mathcal{RSP}(n)$. We are interested in the joint distribution of these statistics.
	Let $a_{n,k,r}$ denote the number of run-sorted permutations over $[n]$ having $k$ runs and $r$ right-to-left minima. If $n=0$, the only nonzero term is $a_{0,0,0}=1$, if $n\geq 1$, then $a_{n,k,r}=0$, where $k>\lceil \frac{n}{2}\rceil, r>n, r<k, k<, r<1$, or $r>n-k+1$.
	\begin{proposition}
		For all integers $n, k, r$ such that $1\leq k, r\leq n$ the numbers $a_{n,k,r}$ of run-sorted permutations over $[n+2]$ having $k$ runs and $r$ right-to-left minima satisfy
		\begin{equation*}
			a_{n+2,k,r}=a_{n+1,k,r-1}+\sum_{i=1}^n{n \choose i}a_{n+1-i,k-1,r-1}.
		\end{equation*}
	\end{proposition}
\begin{proof} 
	Let $\pi\in\mathcal{RSP}(n+2)$. Let us suppose that the integers $1$ and $2$ are in the same run of $\pi$. Let $\pi'$ be the permutation obtained from $\pi$ by deleting $1$ and then decreasing each of the remaining integers by $1$, then $\pi'\in \mathcal{RSP}(n+1)$ and $\run(\pi)=\run(\pi')$ and $\rlmin(\pi)=\rlmin(\pi')+1$. This implies that $$|\{\pi\in\mathcal{RSP}(n+2):~\run(\pi)=k, \rlmin(\pi)=r, 1 \text{ and } 2 \text{ are in the same run}\}|=a_{n+1,k,r-1}.$$  Let us suppose now that $1$ and $2$ are in different runs of $\pi$ and that the first run (containing $1$) has length $i+1, i\geq1$, then we can choose $i$ elements from the set $\{3, 4, \ldots, n+2\}$ to include in the first run. There are $\displaystyle {n \choose i}$ ways to do so. The remaining part of $\pi$ is a run-sorted permutation over $[n-i+1]$ and there are $a_{n+1-i}$ of them. In this case, the number of runs and the number of right-to-left minima of $\pi$ each increase by $1$. This completes the proof.
\end{proof}
	\begin{theorem}
		We have \begin{equation}\label{rrrrlmin} a_{n,k,r}=a_{n-1,k,r-1}+(k-1)a_{n-1,k,r}+(n-2)a_{n-2,k-1,r-1}, n\geq2, k,r\geq 1\end{equation} with the initial conditions $a_{0,0,0}=1, a_{1,1,1}=1$.
	\end{theorem}
	\begin{proof}
		The proof is based on the technique used in the proof of Theorem \ref{rrNaRa}. Let $\pi'$ be a run-sorted permutation over $[n]$ obtained from $\pi\in \mathcal{RSP}(n-1)$ by inserting $n$ at the end of any of its runs. This operation preserves the number of runs. It also preserves the number of right-to-left minima except when $n$ is inserted at the end of the last run of $\pi$, in which case the number of right-to-left minima increases by $1$. So, we get the first two terms of the right-hand side of the recurrence relation. Again, if $\pi'\in \mathcal{RSP}(n)$ is obtained from $\pi\in \mathcal{RSP}(n-2)$ by the operation defined in Lemma \ref{lemmacase2}, that is, $\pi'=\psi(i,\pi)$, where $i\in[n-2]$, then the number of runs and the number of right-to-left minima each increases by $1$. We showed already that this is true for the number of runs, let us show it for the number of right-to-left minima. The operation increases by $1$ each integer greater than $i$ in $\pi$ and inserts the subword $n~i{+}1$ immediately after the rightmost position of the integers of the set $\{1, 2, \ldots, i\}$, then the newly created run beginning at $i+1$ contributes one more right-to-left minimum since the minima of the runs form an increasing subsequence. Thus, we have the last term of the right-hand side of the recurrence.
	\end{proof}
	\begin{theorem}
		The exponential generating function \begin{displaymath} A(x,y,z)=
		\sum_{n,k,r\geq0}a_{n,k,r}\frac{x^n}{n!}y^kz^r\end{displaymath} satisfies the differential equation
		\begin{equation}
			\frac{\partial A}{\partial x}=yze^{xz+yz(-x-1)+ye^x}
			\label{gfrunrlmin}
		\end{equation}
		with the initial condition $\frac{\partial A}{\partial x}|_{x=0}=yz$.
		\label{thmgfrunrlmin}
	\end{theorem}
	\begin{proof} 
		From (\ref{rrrrlmin}) we have \begin{align*}
		\sum_{n\geq2,k,r\geq1}a_{n,k,r}\frac{x^n}{n!}y^kz^r&=\sum_{n\geq2,k,r\geq1}a_{n-1,k,r-1}\frac{x^n}{n!}y^kz^r+\sum_{n\geq2,k,r\geq1}(k-1)a_{n-1,k,r}\frac{x^n}{n!}y^kz^r+\notag\\&~~~~\sum_{n\geq2,k,r\geq1}a_{n-2,k-1,r-1}\frac{x^n}{n!}y^kz^r.\end{align*} Using the notation $\frac{\partial A}{\partial y}=A_y$ and expressing the above equation in terms of $A$ we obtain \begin{align}
			 A&= z\left(\int Adx-x\right)+\int yA_ydx-\left(\int Adx-x\right)+xyz\int Adx-\notag\\&~~~~2yz\int \int Adxdx+1+xyz
			\notag\\
			&=\int yA_ydx-(1-z-xyz)\int A dx+x-xz-2yz\int \int A dxdx+1+xyz.
			\label{gfdoubleint}
 		\end{align}
		By differentiating both sides of (\ref{gfdoubleint}) with respect to $x$ we obtain the following:
		\begin{align*}
			A_x=yA_y-(1-z-xyz)A+1-z-yz\int Adx+yz.
		\end{align*}
		Again by differentiating the above equation with respect to $x$ we obtain
		\begin{align}
			A_{xx}=yA_{yx}-(1-z-xyz)A_x.
			\label{secdiff}
		\end{align}
		By letting $A_x=B$ in (\ref{secdiff}) we obtain
		\begin{align}
			B_x-yB_y+(1-z-xyz)B=0.
			\label{firstdiffsubst}
		\end{align}
		Then, the characteristic equation is $\frac{dy}{dx}=\frac{-y}{1}$ or $\ln y+x=k$ with $k$ constant. We make the transformation with $\epsilon=x, \mu=\ln y+x, \zeta=z$, and  $w(\epsilon,\mu,\zeta)=B(x,y,z)$. Using the substitution we find that (\ref{firstdiffsubst}) transforms to
		\begin{align*}
			w_\epsilon+\left(1-\zeta-\epsilon\zeta e^{\mu-\epsilon}\right)w=0.
		\end{align*}
		By the integrating factor method  we have
		\begin{align*}
			\frac{\partial}{\partial\epsilon}\left(e^{\int\left(1-\zeta-\epsilon\zeta e^{\mu-\epsilon}\right)d\epsilon}w\right)=0
		\end{align*}
		and integrating it with respect to $\epsilon$ and simplifying
		\begin{align*}
			w(\epsilon,\mu,\zeta)&=g(\mu,\zeta)e^{\int\left(1-\zeta-\epsilon\zeta e^{\mu-\epsilon}\right)d\epsilon}\\
			&=g(\mu,\zeta)e^{-\epsilon+\zeta\epsilon+\zeta e^{\mu-\epsilon}(-\epsilon-1)+h(\mu,\zeta)},
		\end{align*}
		where $g$ and $h$ are any differentiable functions of two variables. Using the initial condition $B(0,y,z)=yz$ we have $x=0, \epsilon=0, \mu=\ln y, \zeta=z, w(\epsilon=0,\mu=\ln y,\zeta=z)=yz$, and
		\begin{align*}
			y&=g(\ln y,z)e^{-y+h(\ln y,z)}\\
			ye^y&=g(\ln y,z)e^{h(\ln y,z)}.
		\end{align*} 
		Thus, we obtain $g(t,z)=ze^t, h(t,z)=e^t$. Therefore, we back the transformation in terms of $x, y,z $ so that
		\begin{align*}
			B(x,y,z)&=g(\ln y+x,z)e^{-x+zx+yz(-x-1)+h(\ln y+x,z)}\\
			&=yze^{xz+yz(-x-1)+ye^x}.
		\end{align*}
	 	\end{proof}
By specializing $z=1$ in (\ref{gfrunrlmin}) we obtain the result about the exponential generating function counting run sorted permutations by the number of runs \cite{Na-Ra}. Recall that $r_{n,k}$ is the number of run-sorted permutations over $[n]$ having $k$ runs.
	\begin{corollary}
		If $A(x,y)=\sum_{n,k\geq1}r_{n,k}\frac{x^n}{n!}y^k$, then $A$ satisfies \begin{align*} \frac{\partial A}{\partial x}=ye^{x+y(-x-1)+ye^x}\end{align*}
		with the initial condition $\frac{\partial A}{\partial x}|_{x=0}=y$.
	\end{corollary}
By specializing $y=z=1$ in (\ref{gfrunrlmin}) we obtain the well-known result about the exponential generating function counting the number of run sorted permutations (merging-free partitions) \cite{Bo,St1}.
	\begin{corollary}
	\label{diffgenfun}
	The exponential generating function $A(x)$ of the number of run sorted permutations has the closed differential form
	$$A'(x)=e^{e^x-1}.$$
\end{corollary}
\begin{corollary}
	For all positive integer $n\geq1$, the number $a_n$ of run sorted permutations over $[n]$ is given by
	$$
	a_n=\frac{1}{e}\sum_{m\geq 0}\frac{m^{n-1}}{m!}.$$
\end{corollary}
\subsection{A bijection with separated partitions}
We now consider set partitions with no two consecutive integers in the same block. Such partitions have been studied, for instance, by Munagi \cite{Mu}, who called them ``separated'' partitions and proved that separated partitions over $[n]$ having $k$ blocks are counted by the shifted Stirling numbers of the second kind (see \seqnum{A008277}), like run sorted permutations over $[n]$ having $k$ right-to-left minima. 

It is then natural to provide a bijection between these two equisized classes of objects. Let $\mathcal{P}_n$ denote the set of all separated partitions over $[n]$. 
Let $P=B_1/B_2/\cdots/B_k\in\mathcal{P}_n$ with $k$ blocks.
Define a map $\theta : \mathcal{P}_n\mapsto\mathcal{RSP}(n)$ given by $\pi=\theta(P)$, where $\pi$ is obtained as follows:
\begin{itemize}
	\item[-] for $i=2, \ldots, k$ :
	 
	 ~~~~~		if $b\in B_i, b\neq \min(B_i)$ and $b-1\in B_j$ such that $j<i$, then
	 
	 ~~~~~~~~~		move $b$ to $B_{i-1}$ and rearrange the elements of $B_{i-1}$ in increasing order;
	 \item[-]  flatten the resulting partition and set it to $\pi$.
\end{itemize}
\begin{example}
	If $P=1358/26/47$, then $1358/26/47\rightarrow13568/2/47\rightarrow13568/27/4~\rightarrow13568274=\pi$.
\end{example}
\begin{theorem}
	The map $\theta$ is a bijection.
\end{theorem}
\begin{proof}
	We first prove that $\pi\in\mathcal{RSP}(n)$, for every $P\in \mathcal{P}_n$. The procedure never moves $\min(B_i)$ for all $i$. Thus, the minima remain in increasing order and hence $\pi$ is a run-sorted permutation.
	We now show that if $P$ has $k$ blocks, then $\pi$ has $k$ right-to-left minima.  Obviously, the minimum of each block of $P$ becomes a right-to-left minimum of $\pi$. Let $b$ be in the block $B_i$ with $b\neq \min(B_i)$. The integer $b-1$ is in different block, say $B_j$. If $j<i$, then the procedure moves $b$ to the block $B_{i-1}$ leaving $\min(B_i)$ on its right in $\pi$. Therefore, $b$ cannot be a right-to-left minimum of $\pi$. Suppose that $j>i$. Since $b-1\geq \min(B_j)$ we have $b>\min(B_j)$, so the procedure moves neither $b$ nor $\min(B_j)$ which implies that $b$ cannot be a right-to-left minimum of $\pi$. Therefore, $b$ is a right-to-left minimum of $\pi$ if and only if $b=\min(B_i), i=1, \ldots, k$. We next prove that $\theta$ is one-to-one. Suppose that $P\neq P'$, where $P, P'\in \mathcal{P}_n$. If the number of blocks of $P$ and the number of blocks of $P'$ are different, then we are done since $\theta(P)$ and $\theta(P')$ have different number of right-to-left minima. Let $P=B_1/\cdots/B_k$ and $P'=B'_1/\cdots/B'_k$, and assume that there exists an element $b\in B_i$ and $b\in B'_j$ such that $B_i$ is the block of $P$ and $B'_j$ is the block of $P'$ with $i\neq j$. We take the minimal of these elements. Up to exchanging of $P$ and $P'$ we can suppose $i<j$.
	\begin{enumerate}
		\item If $b=\min(B_i)$, then $b$ is the $i$-th right-to-left minimum of $\theta(P)$ and it would not be the case for $\theta(P')$.
		\item Let $b\neq \min(B_i)$ and let $b-1\in B_m$. Note that $b-1\in B'_m$ for the minimality of $b$. Three sub-cases are possible:
		\begin{itemize}
			\item if $m<i<j$, then $\theta$ moves $b$ to the block $B_{i-1}$ of $P$ and moves $b$ to the block $B'_{j-1}$ of $P'$;
			\item if $i<m<j$, then $\theta$ leaves $b$ in the block $B_i$ in $P$ while it moves $b$ to the block $B_{j-1}'$ in $P'$. Note that $j-1\neq i$, because $i\lneq m\lneq j$;
			\item if $i<j<m$, then  $\theta$ leaves $b$ in the block $B_i$ in $P$ and leaves $b$ in the block $B'_j$ in $P'$.
		\end{itemize}
	\end{enumerate}
	In all cases we have $\theta(P)\neq\theta(P')$.
	Therefore, $\theta$ is a bijection.
\end{proof}
We now present the inverse of $\theta$. Let $\pi\in \mathcal{RSP}(n)$ with $k$ right-to-left minima. We construct $P=\theta^{-1}(P')$ as follows:
\begin{itemize}
	\item[-] insert a slash before each right-to-left minimum  of $\pi$ and let $B_1/\cdots/B_k$ be the resulting partition;
	\item[-] for $i=k, \ldots, 1$:
	 
	 ~~~~~	for $b$ in $B_i\setminus \{\min(B_i)\}$ taken in increasing order
	 
	 ~~~~~~~~~		if $b-1\in B_j, j\leq i$, then
	 
	 ~~~~~~~~~~~~~			move $b$ to $B_{i+1}$ and rearrange the elements in each block in increasing order.
\end{itemize}
It can be easily checked that $\theta^{-1}$ constructs a separated partition.
For instance, if $\pi=13625784$, then by inserting slashes before each right-to-left minimum we have the partition $136/2578/4$, and we move $7$ to $B_3$ since $6\in B_1$. So, $P=136/258/47$.
	\section{Non-crossing merging-free partitions}
	\label{secNoncross}
	A non-crossing partition of a set $A=[n]$ is a partition in which no two blocks ``cross'' each other, that is, if $a$ and $b$ belong to one block and $x$ and $y$ to another, then they cannot be arranged in the order $a x b y$. If one draws an arc connecting $a$ and $b$, and another arc connecting $x$ and $y$, then the two arcs cross each other if the order is $a x b y$ but not if it is $a x y b$ or $a b x y$. In the latter two orders the partition $\{ \{a, b \}, \{x, y\} \}$ is non-crossing  \cite{Si}.
	\begin{example}
		In the following figure, the diagrams of $P=1~2~5~7/~3~9~10/~4~6~8$ and of $P'=1~2~3~9~10/~4~6~7~8/~5$, respectively crossing and non-crossing partitions.
		\begin{figure}[H]
			\centering
			\includegraphics[scale=.7]{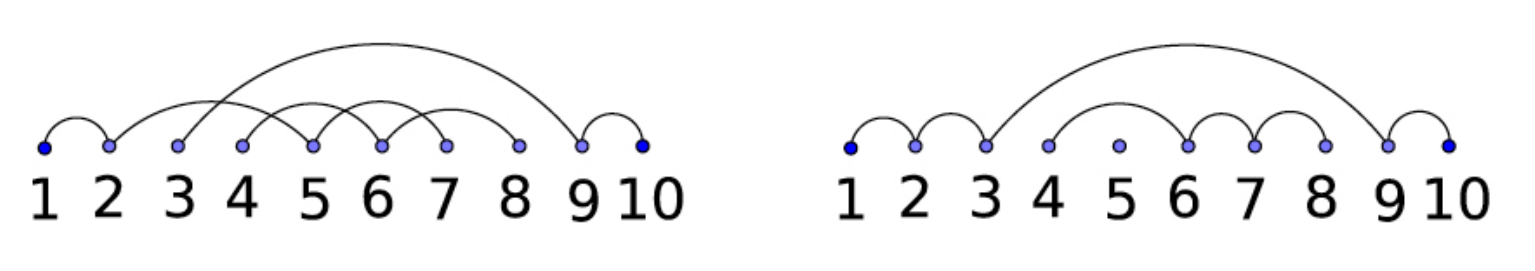}
			\caption{Crossing and non-crossing merging-free partitions}
		\end{figure}
	\end{example}
	We are interested in non-crossing merging-free partitions over $[n]$. We let $\mathcal{M}_n$ denote the set of all non-crossing merging-free partitions over $[n]$, and $\mathcal{M}_{n,t}:=\{P\in M_n : \blocks(P)=t\}$, where $\blocks(P)$ denote the number of blocks of the partition $P$.
	\begin{theorem}
		For all integers $n\geq1$ we have 
		\begin{equation} \label{ncmfp}\sum_{P\in \mathcal{M}_n}q^{\blocks(P)}=\sum_{t=1}^{\frac{n+1}{2}}{n-1 \choose 2(t-1)}q^t.\end{equation}
	\end{theorem}
	\begin{proof}
		We use strong induction on $n$, and provide a recursive construction for the merging-free partitions of $\mathcal{M}_n$. For $n=1$ the assertion is trivially true (initial condition). Assume that $n\geq2$ and the assertion is true for all integers smaller than $n$. We distinguish two cases, depending on if $n$ is in the same block of a merging-free partition $P$ as $n-1$ or not. Suppose that $P$ has $t$ blocks. 
		
		\textbf{Case-1.} If $n$ is in the same block as $n-1$, then we delete $n$ and obtain a non-crossing merging-free partition $P'$ in $\mathcal{M}_{n-1,t}$. The induction hypothesis implies that
		\begin{align}\label{eqsameblock}\sum_{\substack{P\in \mathcal{M}_n\\
		n\text{ and } n-1 \text{ are in}\\\text{the same block of } P}}q^{\blocks(P)}&=\sum_{\substack{P'\in \mathcal{M}_{n-1}}}q^{\blocks(P')}\notag\\
	&=\sum_{t=1}^{\frac{n}{2}}{n-2 \choose 2(t-1)}q^t.\end{align}
		
		\textbf{Case-2.} Suppose that $n$ is not in the same block as $n-1$ in $P$ and, that $n$ is in the same block as a certain $i$, with $i<n-1$. Let $i$ be the maximum of such elements. In this case, as we shall see, all the integers $1, 2, \ldots, i, n$ are in the first block of $P$. Assume for a contradiction that there is $j \in \{1, 2, \ldots, i-1\}$ such that $j$ is not in the same block as $i$. We can choose $j$ such that $j+1$ is in the same block as $i$. If  $j$ is not the maximum element of its block, then the arc relating it to its successor in the partition creates a crossing. 
	\makeatletter
	\def\@captype{figure}
	\makeatother
	\begin{center} 
		\includegraphics[width=0.5\columnwidth]{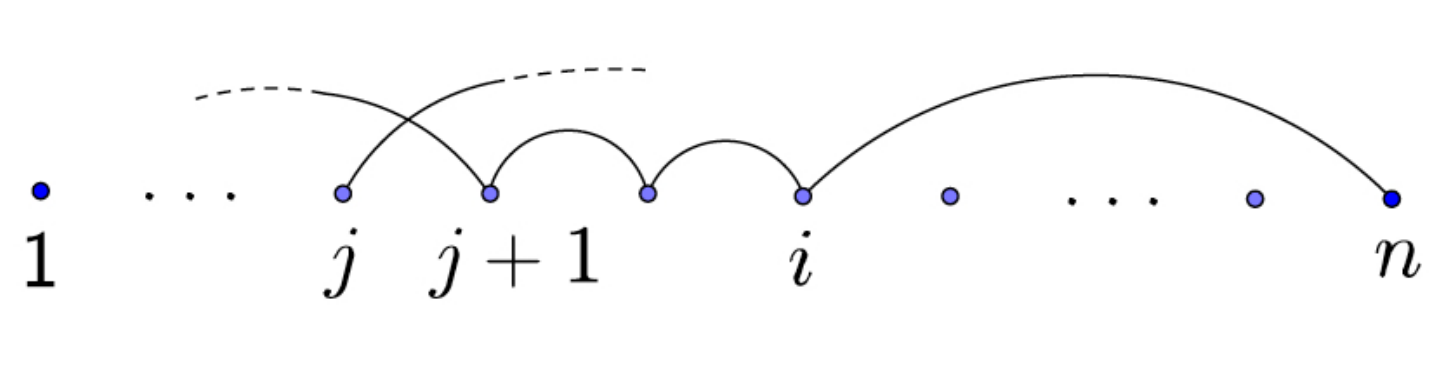} 
		\caption{}\end{center}
		If instead $j$ is the maximum element of its block, then there is an integer $k>i$ such that $k$ is in the same block as an integer $j'<j$ (and hence creates a crossing with the arc $(i, n)$), because otherwise, the block containing $j$ should be merged with one of the blocks containing integers of $[i+1, n-1]$ and hence the partition would not be merging-free.
		\begin{center}
			\makeatletter
			\def\@captype{figure}
			\makeatother
			\includegraphics[width=0.5\columnwidth]{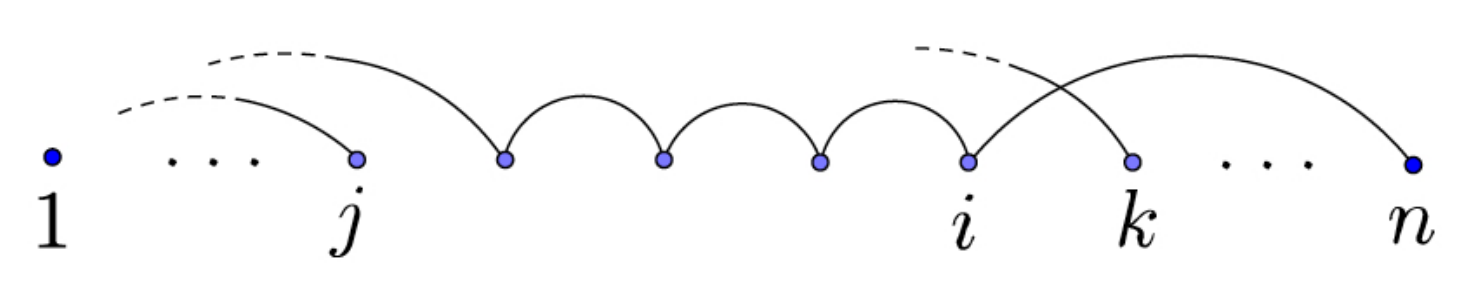}
			\caption{}
		\end{center}
		Thus, the first block is uniquely determined by the integer $i$ and the remaining $n-i-1$ integers must form a non-crossing merging-free partition. So let $P''$ be the partition obtained by deleting the first block of $P$ and then standardizing the resulting partition, that is, subtract $i$ from each of the remaining integers, so we have $P''\in \mathcal{M}_{n-i-1,t-1}$. Thus by the induction hypothesis and taking the sum over all possible $i$ we have
		\begin{align}\label{eqdiffblock}\sum_{\substack{P\in \mathcal{M}_n\\
				n\text{ and } n-1 \text{ are not}\\\text{in the same block of } P}}q^{\blocks(P)}&=\sum_{i=1}^{n-2}q\sum_{\substack{P''\in \mathcal{M}_{n-i-1}}}q^{\blocks(P'')}\notag\\&=\sum_{i=1}^{n-2}q\sum_{t=2}^{\frac{n-i+2}{2}}{n-i-2 \choose 2(t-2)}q^{t-1}.\end{align}
		Therefore, putting (\ref{eqsameblock}) and (\ref{eqdiffblock}) together we have
		\begin{align*}
			\sum_{P\in \mathcal{M}_n}q^{\blocks(P)}
			&=\sum_{t=1}^{\frac{n}{2}}{n-2 \choose 2(t-1)}q^t+\sum_{i=1}^{n-2}q\sum_{t=2}^{\frac{n-i+2}{2}}{n-i-2 \choose 2(t-2)}q^{t-1}\\
			&=\sum_{t=1}^{\frac{n}{2}}{n-2 \choose 2(t-1)}q^t+\sum_{t=2}^{\frac{n+1}{2}}q^t\sum_{i=1}^{n-2t+2}{n-i-2 \choose2t-4}\\
			&=\sum_{t=1}^{\frac{n}{2}}{n-2 \choose 2(t-1)}q^t+\sum_{t=2}^{\frac{n+1}{2}}{n-2 \choose 2t-3}q^t\\
			&=q+\sum_{t=2}^{\frac{n+1}{2}}\left[{n-2 \choose 2t-2}+{n-2 \choose 2t-3}\right]q^t=\sum_{t=1}^{\frac{n+1}{2}}{n-1 \choose 2t-2}q^t.
		\end{align*}
	\end{proof}
	\begin{corollary}
		The number $m_n$ of non-crossing merging-free partitions over $[n]$ is equal to $2^{n-2}$, where $n\geq 2$ and $m_1=m_2=1$.
	\end{corollary}
\begin{proof}
	This follows from putting $q=1$ in (\ref{ncmfp}) and using the well-known identity $$\sum_{k\geq0}{n \choose 2k}=2^{n-1}.$$
\end{proof}
	A function $f\in\RGF(n)$ is said to avoid a pattern $212$ if there does not exist some indices $a<b<c$ such that $f_a=f_c>f_b$. Let $f$ be the canonical form of a set partition $P$ over $[n]$, then $P$ is non-crossing if and only if $f$ is $212$-avoiding (see \cite{Li-Fu,Si1}). A function $f=f_1f_2\cdots f_n$ is said to be \textit{weakly uni-modal} if there exists a value $m\leq n$ for which it is weakly increasing for $i\leq m$ and weakly decreasing for $i\geq m$, that is, $f_1\leq f_2\leq\cdots\leq f_m\geq f_{m+1}\geq\cdots\geq f_n$. Thus, we have the following result.
	\begin{proposition}
		Let $f$ be the canonical form of a merging-free partition $P$, then $f$ is $212$-avoiding if and only if it is weakly uni-modal.
	\end{proposition}
\begin{proof}
	Consider the forward implication. Since $f$ is a $\RGF$, if it is not weakly uni-modal, although it is $212$-avoiding, then it contains either a pattern $312$ or $213$. If $f$ contains a $312$, then before the $3$ there is a $2$ and hence it contains $212$. This is a contradiction. Suppose that $f$ contains a pattern $213$ in the positions $a<a+1<b$, where $b$ is the smallest such integer. If $f_b$ is a left-to-right maximum letter in $f$, then there exists some integer $c>b$ such that $f_c=f_b-1$ because $P$ is merging-free. Since $f$ is a $\RGF$ there exists an integer $d\leq a$ such that $f_d=f_c$ because of the choice of $b$. Thus, $f$ contains the pattern $212$ in $d<a+1<c$ and this is a contradiction. If $f_b$ is not a left-to-right maximum letter, then we have some integer $e<a$ such that $f_e=f_b$, hence $f$ contains the pattern $212$ in $e<a+1<b$ and this is also a contradiction. Therefore, $f$ is weakly uni-modal. The converse implication is clearly true. 
\end{proof}

	
	\section{The Exhaustive Generation}
	\label{secalgo}
	We used the results presented here -- and in particular the construction shown in the proof of the recurrence relation for the sets $\mathcal{RSP}(n,k)$ -- to implement an algorithm to generate these objects, that is, an algorithm which, for any fixed integer $n$, solves the following problem:
	
	\textbf{Problem:} Run-Sorted-Permutations-Generation
	
	\textbf{Input:} an integer $n$
	
	\textbf{Output:} the set of all run-sorted permutations $\mathcal{RSP}(n)$, partitioned into the subsets $\mathcal{RSP}(n,k)$.
	
	Compared to previous, more naive algorithms, this algorithm has allowed some researchers to extend the range of calculations performed in acceptable time and confirm various conjectures for a larger value of $n$. Rather than implementing recursive algorithms we made use of dynamic programming and obtained iterative algorithms. All algorithms implements a partition as a list of integers and a set of partitions as a list of partitions and hence as a list of lists of integers. Algorithm 1 is used to generate run-sorted permutations in $\mathcal{RSP}^{(1)}(n,k)\subseteq \mathcal{RSP}(n,k)$ from $\mathcal{RSP}(n-1,k)$ based on the idea presented in section \ref{sec3}. 
	\begin{algorithm}[H]
		\caption{Exhaustive Generation of Run-Sorted Permutations in $\mathcal{RSP}^{(1)}(n,k)$ from the partitions of $\mathcal{RSP}(n-1,k)$}
		\begin{algorithmic}
			\STATE\textbf{Procedure:} \texttt{FUNCTION\_ONE}$(\mathcal{RSP},n)$
			\ENSURE $\mathcal{RSP}$ is a list of lists whose elements represent run-sorted permutations of size $n-1$, and $n$ is the integer to be inserted so that the number of blocks remains the same.
			\STATE $L\longleftarrow [ ~]$
			
			\FOR{$\pi$ \textbf{in} $\mathcal{RSP}$ }
			\FOR{ $t$ \textbf{in} $Range(Length(\pi)-1)$}
			\IF{ $\pi[t]>\pi[t+1]$ }
			\STATE $\pi'\longleftarrow \pi.Insert(t+1, n)$
			\STATE $L.Append(\pi')$
			\ENDIF
			\ENDFOR
			\STATE $\pi'\longleftarrow\pi.Append(n)$
			\STATE $L.Append(\pi')$
			\ENDFOR 
			\RETURN $L$
		\end{algorithmic}
	\end{algorithm}
	
	The exhaustive generation of run-sorted permutations in $\mathcal{RSP}^{(2)}(n,k)$ starts from the set $[n-2]\times \mathcal{RSP}(n-2,k-1)$, that is, from a pair $(i,\pi)\in [n-2]\times \mathcal{RSP}(n-2,k-1)$ we obtain a run-sorted permutation $\pi'\in \mathcal{RSP}^{(2)}(n,k)$ using Algorithm 2. The idea is based on the operation given in section \ref{sec3}.
	\begin{algorithm}[H]
		\caption{Generation of a run-sorted permutation of $\mathcal{RSP}^{(2)}(n,k)$ from an element of $[n-2]\times \mathcal{RSP}(n-2,k-1)$}
		\begin{algorithmic}
			\STATE\textbf{Procedure:} \texttt{RUN\_SORTED\_PERMUTATION\_SIZE\_INC\_BY\_TWO}($(\pi, p)$)
			\ENSURE $\pi$ is a run-sorted permutation in $\mathcal{RSP}(n-2,k-1)$ and $p$ is an integer in $[n-2]$.
			\FOR{$t$ \textbf{in} $Range(Length(\pi))$}
			\IF {$\pi[t]>p$}
			\STATE $\pi[t]\longleftarrow \pi[t]+1$
			\ENDIF
			\ENDFOR
			\STATE $pos=Length(\pi)-1$
			\WHILE{$\pi[pos]>p$}
			\STATE $pos\longleftarrow pos-1$
			\ENDWHILE
			\STATE $\pi.Insert(pos+1, Length(\pi)+2)$
			\STATE $\pi.Insert(pos+2, p+1)$
			\RETURN $(\pi)$ 
		\end{algorithmic}
	\end{algorithm}
	
	Algorithm 3 calls Algorithm 2 and gives us the exhaustive generation algorithm for the set of run-sorted permutations in $\mathcal{RSP}^{(2)}(n,k)$.
	\begin{algorithm}[H]
		\caption{Exhaustive Generation of Run-sorted Permutations in $\mathcal{RSP}^{(2)}(n,k)$}
		\begin{algorithmic}
			\STATE\textbf{Procedure:} \texttt{FUNCTION\_TWO}$(\mathcal{RSP})$
			\ENSURE $\mathcal{RSP}$ is a list of lists whose elements represent a run-sorted permutation of $\mathcal{RSP}(n-2,k-1)$.
			\STATE $L\leftarrow [ ~]$
			\FOR{$\pi$ \textbf{in} $\mathcal{RSP}$}
			\FOR{$p$ \textbf{in} $Range(1, Length(\pi)+1)$}
			\STATE $\pi'\longleftarrow$ \texttt{RUN\_SORTED\_PERMUTATION\_SIZE\_INC\_BY\_TWO}$(\pi,p)$
			\STATE $L.Append(\pi')$
			\ENDFOR
			\ENDFOR
			\RETURN $L$
		\end{algorithmic}
	\end{algorithm}
	
	Now we present the main exhaustive generation algorithm that generates all and only those run-sorted permutations in $\mathcal{RSP}(n)$ for all possible $n$. The algorithm returns a list of lists of lists of integers, namely the list $\mathcal{RSP}(n)=[\mathcal{RSP}(n,1), \mathcal{RSP}(n,2), \ldots, \mathcal{RSP}(n,n)]$ where each element $\mathcal{RSP}(n,k)$ is the list of all run-sorted permutations of $[n]$ having $k$ runs.
	Since $\mathcal{RSP}(n,k)=\varnothing$ if $k>\lceil{\frac{n}{2}}\rceil$, the algorithm can be optimized by competing only those sets $\mathcal{RSP}(n,k)$ that are not empty.
	As we said, the algorithm is based on dynamic programming. We stock the values of the lists $\mathcal{RSP}(n-1)$ and $\mathcal{RSP}(n-2)$ and use them to compute the list $\mathcal{RSP}(n)$.
	In order to save memory, only the last two lists are kept at any time: the list $\mathcal{RSP}(n-1)$ will be stocked in a variable called \texttt{LastRow} and the list $\mathcal{RSP}(n-2)$ will be stocked in a variable called \texttt{RowBeforeLast}, while list $\mathcal{RSP}(n)$ will be affected to the variable \texttt{CurrentRow}.
	At the end of each iteration, the three variables are shifted.
	\begin{algorithm}[H]
		\caption{Exhaustive Generation of Run-Sorted Permutations}
		\begin{algorithmic}
			\STATE\textbf{Procedure:} \texttt{RUN\_SORTED\_PERMUTATIONS}($n$)
			\ENSURE $\mathcal{RSP}$ is a list of lists whose elements represent a run-sorted permutation of $\mathcal{RSP}(n-2,k-1)$.
			\STATE \texttt{RowBeforeLast}$\longleftarrow[[[1]]]$
			\STATE \texttt{LastRow}$\longleftarrow[[[1,2]],[~]]$
			\FOR{$i$ \textbf{in} $Range(3,n+1)$}
			\STATE \texttt{CurrentRow}$\longleftarrow[~]$
			\STATE \texttt{CurrentRow}$.Append$(\texttt{FUNCTION\_ONE}(\texttt{LastRow}$[0],i$))
			\FOR{$j$ \textbf{in} $Range(1,\lceil\frac{i}{2}\rceil)$}
			\STATE \texttt{CurrentRow}$.Append($\texttt{FUNCTION\_ONE}(\texttt{LastRow}$[j],i$)
			\STATE \texttt{CurrentRow}$.Append($\texttt{FUNCTION\_TWO}$($\texttt{RowBeforeLast}$[j-1]))$
			\ENDFOR
			\STATE \texttt{RowBeforeLast}$\longleftarrow$ \texttt{LastRow}
			\STATE \texttt{LastRow}$\longleftarrow$ \texttt{CurrentRow}
			\ENDFOR
		\end{algorithmic}
	\end{algorithm}
	
	All these algorithms have been implemented in Python.

	\begin{example}
		When \texttt{RUN\_SORTED\_PERMUTATIONS}($n$) is executed for $n=5$  we get the list $\mathcal{RSP}(5)$: 
		\begin{align*}
			&[[[1,2,3,4,5]],\\
			&[[1,3,4,5,2], [1,3,4,2,5], [1,3,5,2,4], [1,3,2,4,5], [1,4,5,2,3], [1,4,2,3,5], [1,2,4,5,3], \\
			&[1,2,4,3,5], [1,5,2,3,4], [1,2,5,3,4], [1,2,3,5,4]],\\
			& [[1,5,2,4,3], [1,4,2,5,3], [1,3,2,5,4]], \\
			&[~]].\end{align*}
		Observe that $a_{5,1}=1, a_{5, 2}=11, a_{5, 3}=3, a_{5,4}=0$, and $a_5=1+11+3+0=15$.
	\end{example}
	
	\section{Acknowledgements}
	The first author is grateful for the financial support extended by IRIF, the cooperation agreement between International Science Program (ISP) at Uppsala University and Addis Ababa University, and the support by the Wenner-Gren Foundations. We appreciate the hospitality we got from IRIF during the research visits of the first author. We also thank our colleagues from CoRS (Combinatorial Research Studio) for valuable discussions and comments, and in particular we thank Dr.\ Per Alexandersson and Prof.\ J\"orgen Backelin for their useful discussions and suggestions.
	
	
\bigskip
\hrule
\bigskip
2010 {\it Mathematics Subject Classification}:
Primary 05A05; Secondary 05A15, 05A19.
 
\emph{Keywords:} merging-free partition, canonical form, run-sorted permutation, non-crossing partition, algorithm.

\end{document}